\documentclass[a4 paper,12pt]{article}
\usepackage[dvips]{epsfig}
\usepackage[T1]{fontenc}
\usepackage{amssymb,amsmath,graphicx,color,amsthm}
\usepackage{centernot}
\usepackage[font=normal,labelfont=bf,width=12.5cm]{caption}

\newtheorem{theorem}{Theorem}
\newtheorem{lemma}[]{Lemma}

\newtheorem{conjecture}{Conjecture}

\newcommand{\set}[1]{\ensuremath{\left\{#1 \right\}}}

\begin{document}

\title{Strong edge coloring of subcubic bipartite graphs}
\author
{
	Borut Lu\v{z}ar\footnotemark[2], Martina Mockov\v{c}iakov\'{a}\footnotemark[1],
	Roman Sot\'{a}k\thanks{Institute of Mathematics, Faculty of Science, Pavol Jozef \v Saf\'arik University, Ko\v sice, Slovakia.			
			E-Mails: \texttt{martina.mockovciakova@upjs.sk, roman.sotak@upjs.sk}},
	Riste \v{S}krekovski\thanks{Faculty of Information Studies, Novo mesto, Slovenia \& Institute of Mathematics, Physics and Mechanics, Ljubljana, Slovenia.  
			E-Mails: \texttt{\{borut.luzar,skrekovski\}@gmail.com}}
}

\maketitle

{
\begin{abstract}
\noindent
	A \textit{strong edge coloring} of a graph $G$ is a proper edge coloring in which each color class is an induced matching of $G$.
	In 1993, Brualdi and Quinn Massey~\cite{BruQui93} proposed a conjecture that every bipartite graph without $4$-cycles and with the maximum 
	degrees of the two partite sets $2$ and $\Delta$ admits a strong edge coloring with at most $\Delta+2$ colors. We prove that this conjecture holds 
	for such graphs with $\Delta=3$. We also confirm the conjecture proposed by Faudree et al.~\cite{FauGyaSchTuz90} for subcubic 
	bipartite graphs.
	
	\bigskip
	NOTE: After publishing this paper on ArXiv, we have been notified that an equivalent result has already been proven by Maydanskiy~\cite{May05}
	in the scope of incidence colorings. Moreover, the latter result has been verified in a little stronger form in~\cite{WuLin08} by Wu and Lin.
\end{abstract}
}


\section{Introduction}

A \textit{strong edge coloring} of a graph $G$ is a proper edge coloring in which each color class is an induced matching of $G$; i.e., any two edges 
at distance at most two are assigned distinct colors. The minimum number of colors for which a strong edge coloring of $G$ exists is the 
\textit{strong chromatic index} of $G$, denoted $\chi^\prime_s (G)$.

In 1985, Erd\H{o}s and Ne\v{s}et\v{r}il proposed the following conjecture during a seminar in Prague.
\begin{conjecture}[Erd\H{o}s, Ne\v{s}et\v{r}il, 1985]
	\label{con:basic}
	Let $G$ be a graph with maximum degree $\Delta$. Then
	$$
		\chi_s'(G) \le \left \{ \begin{array}{ll}
								\frac{5}{4} \Delta^2\,, &\quad \Delta \textrm{ is even;} \vspace{0.3cm}\\ 
								\frac{1}{4}(5\Delta^2 - 2\Delta +1)\,, &\quad \Delta \textrm{ is odd.}
								\end{array} \right.
	$$
\end{conjecture}
Molloy and Reed~\cite{MolRee97} established currently the best known upper bound for the 
strong chromatic index of graphs with sufficiently large maximum degree.
\begin{theorem}[Molloy, Reed, 1997]
	For every graph $G$ with sufficiently large maximum degree $\Delta$ it holds that
	$$
		\chi_s'(G) \le 1.998\,\Delta^2.
	$$
\end{theorem}
Conjecture~\ref{con:basic} was proved for subcubic graphs by Andersen~\cite{And92} and independently by Hor\'{a}k, He and Trotter~\cite{HorHeTro93}.
Some work has been done also for graphs with maximum degree four~\cite{Cra06,Hor90}, but the conjecture is still open even for them.

In 1990, Faudree et al.~\cite{FauGyaSchTuz90} proposed a conjecture for bipartite graphs. 
The complete bipartite graph $K_{\Delta,\Delta}$ shows that the bound, if true, is the best possible.
	\begin{conjecture}[Faudree et al., 1990]
	\label{conj:strongbip}
		Let $G$ be a bipartite graph. Then $\chi^\prime_s(G)\le \Delta^2$.
	\end{conjecture}
\noindent
%Notice that Mahdian in \cite{Mah00} conjectured that the generalization of previous conjecture holding for every $C_5$-free graph is also true.

The authors in \cite{FauGyaSchTuz90} also considered the graphs with all cycle lengths divisible by four. Such graphs are bipartite and none 
of their cycles has a chord. Conjecture~\ref{conj:strongbip} holds for such graphs, furthermore, the authors conjectured that their 
strong chromatic index is bounded by a function linear in $\Delta$.

Regarding the small values of maximum degree $\Delta$, the conjectured bound is trivially true for isolated edges, paths and even cycles. 
In 1993, Steger and Yu~\cite{SteYu93} showed that Conjecture~\ref{conj:strongbip} holds also for subcubic graphs. Note that this conjecture 
is still open for $\Delta\ge 4$.

Conjecture~\ref{conj:strongbip} was generalized by Brualdi and Quinn Massey~\cite{BruQui93} in 1993 as follows.
\begin{conjecture}[Brualdi, Quinn Massey, 1993]
	\label{conj:strongbip2}
		Let $H$  be a bipartite graph with bipartition $X$ and $Y$ such that $\Delta(X) = \Delta_1$ 
		and $\Delta(Y) = \Delta_2$. Then $\chi^\prime_s(G)\le \Delta_1\Delta_2$.
\end{conjecture}
They proved that Conjecture~\ref{conj:strongbip2} is true for graphs with all cycle lengths divisible by four.
Note that Steger and Yu solved this problem in case when $\Delta_1=\Delta_2=3$.
Brualdi and Quinn Massey~\cite{BruQui93} considered graphs with $\Delta_1=2$ and established Conjecture~\ref{conj:strongbip2} for 
graphs without any $4$-cycles.
\begin{theorem}[Brualdi, Quinn Massey, 1993]
	Let $H$ be a bipartite graph with bipartition $X$ and $Y$ without any $4$-cycle such 
	that $\Delta(X) = 2$ and $\Delta(Y) = \Delta$. Then $\chi^\prime_s(H)\le 2\Delta$.
\end{theorem}
In fact, they conjectured that $2\Delta$ in above theorem can be replaced by $\Delta+2$.
\begin{conjecture}[Brualdi, Quinn Massey, 1993]
	\label{conj:2Delta}	
	Let $H$ be a bipartite graph with bipartition $X$ and $Y$ without any $4$-cycle such 
	that $\Delta(X) = 2$ and $\Delta(Y) = \Delta$. Then $\chi^\prime_s(H)\le \Delta+2$.
\end{conjecture}
In 2008 Nakprasit~\cite{Nak08} settled Conjecture~\ref{conj:strongbip2} of Brualdi and Quinn Massey for $\Delta_1=2$.
\begin{theorem}[Nakprasit, 2008]
	\label{nakprasit}
	Let $H$ be a bipartite graph with bipartition $X$ and $Y$ such 
	that $\Delta(X) = 2$ and $\Delta(Y) = \Delta$. Then $\chi^\prime_s(G)\le 2\Delta$.
\end{theorem}
This theorem is, together with the result of Steger and Yu, so far the only known partial confirmation of Conjecture~\ref{conj:strongbip2}.
Here, let us mention that Nakprasit's result follows also from the results of Petersen~\cite{Pet1891} and Hanson, Loten, and Toft~\cite{HanLotTof98}.
\begin{theorem}[Petersen, 1891]
	\label{even2factor}
 Every regular multigraph of positive even degree has a $2$-factor.
\end{theorem}
\begin{theorem}[Hanson, Loten, Toft, 1998] 
	\label{odd2factor}
	For $r,k \in \mathbb{N}$ with $r \ge (3k - 1)/2$, every $(2r + 1)$-regular graph with at most $(2r + 1 - k)/k$ cut-edges has a $2k$-factor.
\end{theorem}

\begin{proof}[Short proof of Theorem~\ref{nakprasit}.]
	We proceed in two steps. Firstly, applying an induction on the maximum degree of a graph, and secondly using a contradiction by showing that
	the minimal counterexample admits the desired coloring.
	
	Let $G$ be a minimal biregular counterexample to the theorem in terms of the number of vertices. Since every bipartite graph 
	is a subgraph of a biregular graph, we consider only biregular graphs. Moreover, bridges in bipartite graphs with maximum degree in 
	one of the partition sets at most $2$ are reducible for a strong edge coloring with at most $2 \Delta$ colors. 
	Hence, the graph $\hat{G}$, $G$ with all the $2$-vertices suppressed, has a $2$-factor $\hat{F}$. Let $F$ be the set of edges corresponding to the edges
	of $\hat{F}$. Then, by induction, $G - F$ admits a strong edge coloring $\varphi_1$ with at most $2\Delta - 4$ colors, 
	furthermore, every cycle in $F$ admits a strong edge coloring with at most $4$ colors. Notice that the edges of the two cycles in $F$ 
	are at distance at least $3$ in $G$, 
	hence there is a strong edge coloring $\varphi_2$ of $F$ with at most $4$ colors. Thus, $\varphi_1$ and $\varphi_2$ induce a strong edge coloring of $G$
	with at most $2\Delta$ colors, a contradiction.
\end{proof}

There are several further open problems regarding subcubic bipartite graphs.
Faudree et al.~\cite{FauGyaSchTuz90} considered also subcubic bipartite graphs and suggested the problems listed below:
\begin{conjecture}[Faudree et al., 1991]
	\label{conj:cubbip}
	Let $G$ be a subcubic bipartite graph. Then
	\begin{itemize}
		\item[$(a)$] $\chi^\prime_s(G) \le 6$, if for each edge $uv \in E(G)$, $d(u)+d(v) \le 5$;
		\item[$(b)$] $\chi^\prime_s(G)\le 7$, if $G$ is $C_4$-free;
		\item[$(c)$] $\chi^\prime_s(G)\le 5$, if the girth of $G$ is large enough.
	\end{itemize}
\end{conjecture}

In this paper we solve Conjecture~\ref{conj:2Delta} for the case of subcubic graphs and the case $(a)$ of Conjecture~\ref{conj:cubbip}.

\bigskip
Throughout the paper, we mostly use the terminology used in~\cite{BonMur08}. We refer to vertices of degree $k$ as \textit{$k$-vertices},
and to cycles of length $k$ as \textit{$k$-cycles}. 
We say that a pair of vertices of a graph $G$ is \textit{$2$-adjacent} if they have a common neighbor in $G$.
A \textit{$2$-neighbor} of a vertex $v$ is any vertex at distance at most $2$ from $v$. Similarly,
a $2$-neighbor of an edge $e$ is any edge at distance at most $2$ from $e$.
By $B_{2,\Delta}$ we denote the set of all bipartite graphs with maximum degrees of partite sets bounded by $2$ and $\Delta$, respectively.

\section{Subcubic bipartite graphs without $4$-cycles}

In this section we consider subcubic bipartite graphs without $4$-cycles.
In particular, we prove that Conjecture~\ref{conj:2Delta} holds if $\Delta=3$. In the proof of the main theorem of this section 
we will make use of the following results.
\begin{lemma}
	\label{lem:str:cover}
	Let $G$ be a graph without isolated vertices. Then, there exists a set of disjoint stars $\mathcal{S}$ whose edges cover all the vertices of $G$, i.e.,
	every vertex in $G$ is a vertex of exactly one star $S$ in $\mathcal{S}$.
\end{lemma}
A cover described in Lemma~\ref{lem:str:cover} can be obtained e.g. by recursively removing the middle edge of every path 
of length $3$ in a spanning tree of a graph. Note that each minimal set of edges which covers the vertices of $G$ is of desired type.
The following lemma is a corollary of Petersen's theorem that every cubic graph with at most one bridge has a $1$-factor~\cite{Pet1891}.
\begin{lemma}[Petersen, 1891]
	\label{lem:str:pet}
	Let $G$ be a simple subcubic graph without bridges and with at most one vertex of degree at most two. 
	Then, there is a matching that covers all $3$-vertices.
\end{lemma}

Notice that since there is always an even number of odd vertices in a graph, an eventual $2$-vertex from the above lemma is never covered.
Now, we prove a theorem that confirms Conjecture~\ref{conj:2Delta} for subcubic graphs. 
\begin{theorem}
	\label{thm:str:bip23}
	Let $G \in B_{2,3}$ be a graph without $4$-cycles. Then
	$$
		\chi_s'(G) \le 5\,.
	$$	
\end{theorem}

The bound in Theorem~\ref{thm:str:bip23} is tight. Consider for an example the Petersen's graph with all edges subdivided in Fig.~\ref{fig:peter}.
It does not admit a strong edge coloring with at most $4$ colors, since it contains a $10$-cycle with a pendent edge on every second vertex, for which
it is not hard to see that it is the only cycle with pendent edges in $B_{2,3}$ that needs at least $5$ colors for a strong edge coloring.
\begin{figure}[ht]
	$$
		\includegraphics{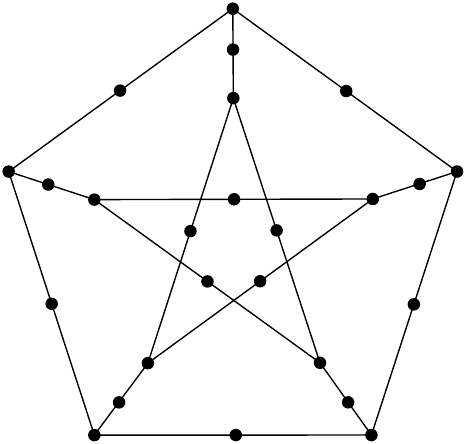}
	$$
	\caption{Petersen's graph with all edges subdivided is bipartite subcubic graph with girth $10$ and strong chromatic index equal to $5$.}
	\label{fig:peter}
\end{figure}

\begin{proof}
Suppose, to the contrary, that the theorem is false. Let $G$ be a minimal counterexample to the theorem with the bipartition $(X,Y)$ and $\Delta(X) = 2$.
By minimal we mean the graph with the minimum number of vertices in $Y$ and, among such, minimum number of $2$-vertices. 

Now we consider some structural properties of $G$. Obviously, $G$ is connected. Suppose first that there is a bridge $uv$ in $G$. 
Let $G_u$ and $G_v$ be the two components of $G - uv$ containing the vertices $u$ and $v$, respectively. 
By the minimality of $G$, there exist strong edge colorings $\varphi_1$ and $\varphi_2$ of $G_u$ and $G_v$ with at most $5$ colors, respectively. 
It is easy to verify that there exists a permutation of colors in $\varphi_1$ such that $\varphi_1$ and $\varphi_2$ induce a strong edge coloring of $G - uv$
and there is an available color for $uv$. Therefore, $G$ is bridgeless and the minimum degree of $G$ is $2$, as the edge incident to a $1$-vertex is a bridge.

Denote by $\hat{G}$ the graph obtained from $G$ by removing every $2$-vertex from $X$ and connecting its two neighbors by an edge. 
Since $G$ has no $4$-cycles, $\hat{G}$ is a simple subcubic graph. Moreover, since $G$ is bridgeless, so is $\hat{G}$.

Suppose that $\hat{G}$ contains two $2$-vertices, $x$ and $y$. Then, $x$ and $y$ are adjacent in $\hat{G}$ for otherwise
we may connect them by a subdivided edge in $G$ and so decrease the number of $2$-vertices in $G$ (what would contradict the minimality assumption). 
But in this case, we color the edges of the graph $G \setminus \set{x,y}$ strongly
with at most $5$ colors and then extend the coloring to $G$ (such an extension is trivial and we leave it to the reader).
Hence, we may assume that $\hat{G}$ contains at most one $2$-vertex. 

By Lemma~\ref{lem:str:pet}, there exists a matching $\hat{M}$ covering all the $3$-vertices of $\hat{G}$.
By removing all the edges of the matching $\hat{M}$ we obtain a $2$-regular graph, i.e., $\hat{C} = \hat{G} \setminus \hat{M}$ is a union of cycles.
Let $G^\star$ be the graph obtained from $\hat{G}$ by contracting all the cycles in $\hat{C}$, i.e., 
we identify all the vertices of each cycle and remove the edges of the cycles. 
Hence, $G^\star$ is a connected multigraph with possible loops. By Lemma~\ref{lem:str:cover}, there exists a set of disjoint stars $S^\star$ in $G^\star$ 
whose edges cover all the vertices of $G^\star$ (in case when $|V(G^\star)| = 1$, $S^\star = \emptyset$). Let $\hat{S} \subseteq \hat{M}$ be the set of edges in $\hat{G}$ corresponding to the edges of $S^\star$.
We call a cycle in $\hat{G}$ associated to a central vertex of a star in $S^\star$ a \textit{central cycle}, and similarly, we refer to a cycle
in $\hat{G}$ corresponding to a leaf of a star in $G^\star$ as a \textit{leaf cycle}.

Now, consider $G$ again. Every edge of $\hat{M}$, $\hat{C}$, and $\hat{S}$ is subdivided once in $G$; denote by $M$, $C$, and $S$ the corresponding sets of edges in $G$.
In the sequel we describe how a strong edge coloring of $G$ with at most $5$ colors can be constructed. Firstly, we color almost all edges of the cycles in $C$ 
and all edges in $S$ with the colors $1$, $2$, and $3$ and show that the remaining noncolored edges in $G$ can be colored by at most two additional colors. 
In particular, after coloring the edges with the colors $1$, $2$, and $3$, we consider the graph $R(G)$ whose vertices are yet noncolored edges of $G$, 
and two vertices are connected if the corresponding edges are at distance at most $2$ in $G$. We refer to $R(G)$ as a \textit{conflict graph}. 

Obviously, a proper coloring of the vertices of $R(G)$ is a strong edge coloring of the corresponding edges in $G$.
Our aim is to show that the graph $R(G)$ is bipartite and hence $2$-colorable. 

Constructing the coloring, we consider two cases regarding the number of the cycles in $C$:
\begin{itemize}
	\item[$(a)$] \textit{The set $C$ contains exactly one cycle $C_0$.}\,  Let $v_1, v_2,\dots, v_k$ be the consecutive vertices of $C_0$ from the partition $Y$,
	i.e., the $3$-vertices of $G$ with an eventual $2$-vertex. Since $G$ is bipartite, the length of $C_0$ is $2k$.
	Choose the natural orientation of $C_0$ induced by the labeling of the vertices. We label the incoming and outgoing edge incident to a vertex 
	$v_t$, $1 \le t \le k$, as $e_t^+$ and $e_t^-$, respectively.
	
	Since $G$ is a minimal counterexample with respect to the number of $2$-vertices, $C_0$ has at least one subdivided chord in $M$, 
	represented by the edges $f_i$ and $f_j$, where $f_i$ is incident to $v_i$, for some $i, 1\le i\le k-2$, and $f_j$ is incident to $v_j$, 
	for some $j, i+2\le j\le k$. By a \textit{subdivided chord} we mean the two edges in $G$ corresponding to a chord of the cycle in $\hat{G}$.
	If possible, we always choose a subdivided chord which does not divide $C_0$ into two cycles, where one is of length $6$.
	
	Next, color the edges of $C_0$ repetitively by the colors $1,2,3$ starting with the edge $e_{i+1}^+$ and continue with the edge $e_{i+1}^-$. 
	If $2k$ is divisible by $3$, all the edges of $C_0$ are strongly colored. Notice that the conflict graph is a union of $\lfloor k/2 \rfloor$ 
	copies of $K_2$, and hence bipartite.
	
	So, suppose that $2k$ is not divisible by $3$. Uncolor the edges $e_i^+$, $e_i^-$, $e_j^+$ and $e_j^-$. The edge $f_i$ has
	at most two colored $2$-neighbors, hence we can color it with some color from $\set{1,2,3}$. Now we color $f_j$, which has three colored $2$-neighbors,
	$f_i$, $e_{j-1}^-$, and $e_{j+1}^+$. However, the edges $e_{j-1}^-$, and $e_{j+1}^+$ are assigned the same color, therefore we can color $f_j$ with some
	color from $\set{1,2,3}$ also. 
	
	Now consider the conflict graph of $G$. We show that it contains no cycle of odd length and hence its vertices are properly colorable with two 
	colors. In case when $C_0$ is an $8$-cycle with two subdivided chords, then $R(G)$ is comprised of two $4$-cycles with a common edge, and hence it is bipartite.
	
	In case when $C_0$ is of length at least $10$ and every its subdivided chord divides $C_0$ into a $6$-cycle and a cycle of higher length,
	then $R(G)$ is comprised of a $4$-cycle and a path with a common vertex, and a disjoint union of $\lfloor \frac{k-4}{2} \rfloor$ copies of $K_2$.
	Otherwise, if no subdivided chord divides $C_0$ into a $6$-cycle, $R(G)$ is comprised an $8$-cycle or a path on $10$ vertices, and some copies of $K_2$,
	hence, it is again bipartite.	
	\begin{figure}[ht]
		$$
			\includegraphics{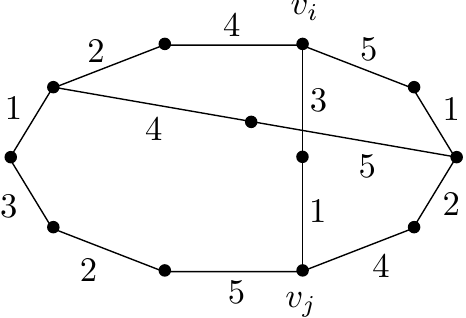}
		$$
		\caption{An example of a strong edge coloring in case when $G$ is comprised of only one cycle with some diagonals.}
		\label{fig:caseA}
	\end{figure}
	
	\item[$(b)$] \textit{The set $C$ contains at least two cycles.}\,  Consider an arbitrary star in $S^\star$ with a central vertex $c$ of degree $p$. 
	Let $C_c$ be the central cycle in $G$ corresponding to $c$. As in the previous case, let the length of $C_c$ be $2k$ and let $v_1,v_2,\dots, v_k$ be its 
	vertices from the partition $Y$. There are $p$ leaf cycles around $C_c$, for $1\le p\le k$. A leaf cycle of $C_c$ is called $C_i$ if one of the two edges
	of $S$, connecting $C_c$ and $C_i$, is incident to $v_i$. We say that the leaf cycles $C_i$ and $C_{i+1}$ are \textit{adjacent}. 	
	Furthermore, let $f_i$ and $g_i$ be the two edges of $S$ connecting the cycle $C_c$ and $C_i$ such that $f_i$ is incident to $v_i$ and consecutively 
	adjacent to $e_i^+$ and $e_i^-$, and $g_i$ is adjacent to two edges of $C_i$.

	As in the previous case, we color all the edges of $C_c$ by a repetitive sequence of colors $1$, $2$, and $3$. 
	If there exists at least one pair of adjacent leaf cycles $C_i$ and $C_{i+1}$, $1 \le i < k$, we start to color the edges of $C_c$ with 
	the edge $e_{i+1}^-$. Then, we uncolor the edge $e_i^-$ and every edge $e_j^+$, for every $j \neq i$ for which there is a leaf cycle $C_j$.	
	In case when there is no pair of adjacent leaf cycles around $C_c$, we start to color $C_c$ with the edge $e_{i+1}^+$, for an arbitrary $i$, and uncolor every edge 
	$e_j^+$ for which there is a leaf cycle $C_j$; if $2k$ is not divisible by $3$, we uncolor also the edge $e_i^-$. 

	Now we color the edges $f_j$ and $g_j$ for every leaf cycle $C_j$ of $C_c$. There is at least one available color in $\set{1,2,3}$ for the edge $f_j$, 
	since we can use the color of the adjacent uncolored edge, so we color $f_j$. The edge $g_j$ has at most two colored $2$-neighbors, hence we can color it also
	by one of the colors $1$, $2$, or $3$.
	Finally, we color the edges of the leaf cycles of $C_c$, except for the two edges adjacent to the edge $g_j$, alternately with the colors $1$, $2$, and $3$. 
	Since $g_j$ is already colored, we may have to permute the colors in order to obtain a partial strong edge coloring.
	
	We repeat the described procedure for every central vertex in $S^\star$ and obtain a partial strong edge coloring of $G$ with $3$ colors. It remains to show
	that the conflict graph $R(G)$ of the noncolored edges of $G$ is bipartite. It is easy to see that every component in the graph $N(G)$ induced by the noncolored edges of $G$
	is an edge (the uncolored edges on the central cycles) or a $2$-path (the edges of $M \setminus S$, the pairs of edges incident to the edges $g_i$, and the pairs of edges 
	incident to the edges $f_i$ and $f_{i+1}$ in cases of adjacent leaf cycles, or the pairs of edges incident to $f_i$ otherwise). 
	
	Notice that vertices in $R(G)$ corresponding to single edge components of $N(G)$ do not lie on any cycle in $R(G)$. Moreover, the maximum
	degree of $R(G)$ is at most $2$, therefore we have that $R(G)$ contains no odd cycle, and so it is bipartite. Thus, we can color the noncolored edges of $G$ 
	by two additional colors. By that, we showed that there exists a strong edge coloring of $G$ with at most $5$ colors, and so establish the theorem.
\end{itemize}
\end{proof}

%%%%%%%%%%%%%%%%%%%%%%%%%%%%%%%%%%%%%%%%%%%%%%%%%%%%%%%%%%%%%%%%%%%%%%%5
\section{Subcubic bipartite graphs with bounded edge weights}

Using Theorem~\ref{thm:str:bip23}, we are able to solve the case $(a)$ of Conjecture~\ref{conj:cubbip} by Faudree et al..
\begin{theorem}
	\label{thm:str:weight}
	Let $G$ be a bipartite subcubic graph with $d(u)+d(v)\le5$ for every edge $uv$. Then
	$\chi_s'(G) \le 6$.
\end{theorem}
Note that the bound is tight and is attained e.g. by the complete bipartite graph $K_{2,3}$.
\begin{figure}[ht]
	$$
		\includegraphics{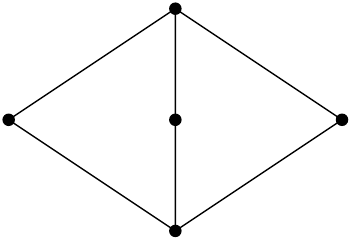}
	$$
	\caption{The strong chromatic index of $K_{2,3}$ is $6$.}
	\label{fig:str:weight}
\end{figure}

\begin{proof}
	Suppose, to the contrary, that $G$ is a counterexample to the theorem with the minimum number of vertices. We firstly discuss
	some structural properties of $G$. Notice that every $3$-vertex in $G$ has three neighbors of degree at most $2$.
	The proof will be an easy consequence of the following three claims:
	\begin{itemize}
		\item[$(a)$] \textit{the minimum degree $\delta(G) = 2$; and}
		\item[$(b)$] \textit{there are no adjacent $2$-vertices in $G$; and}
		\item[$(c)$] \textit{there are no $4$-cycles in $G$.}
	\end{itemize}
	We prove each of the claims separately. 
	\begin{itemize}
		\item[$(a)$] This property simply follows from the minimality of $G$.
		%and the fact that every $1$-vertex has at most four $2$-neighbors.
		
		\item[$(b)$] Suppose first that $u$ and $v$ are adjacent $2$-vertices. By the minimality of $G$, the graph $G \setminus \set{u,v}$
		admits a strong edge coloring $\varphi$ with at most $6$ colors. Let $x$ and $y$ be the second neighbors of $u$ and $v$, respectively.
		We will show that $\varphi$ can always be extended to the edges $ux$, $uv$, and $vy$ and hence to $G$. 
		
		The number of colored $2$-neighbors of the edges $ux$, $uv$, and $vy$ is always at most $4$, hence there are two available colors 
		for each of them. In case when the union of the available colors of all three edges is at least three, $\varphi$ is easily extended. 
		Thus, we may assume that the same two colors, say $5$ and $6$, are available
		for all three edges. Notice that in this case $x$ and $y$ are $3$-vertices and the union of the colors of the colored edges incident to them
		is of size $4$. Consider the labeling of the vertices and the assignment of the colors as in Figure~\ref{fig:str:w1}.
			\begin{figure}[ht]
				$$
					\includegraphics{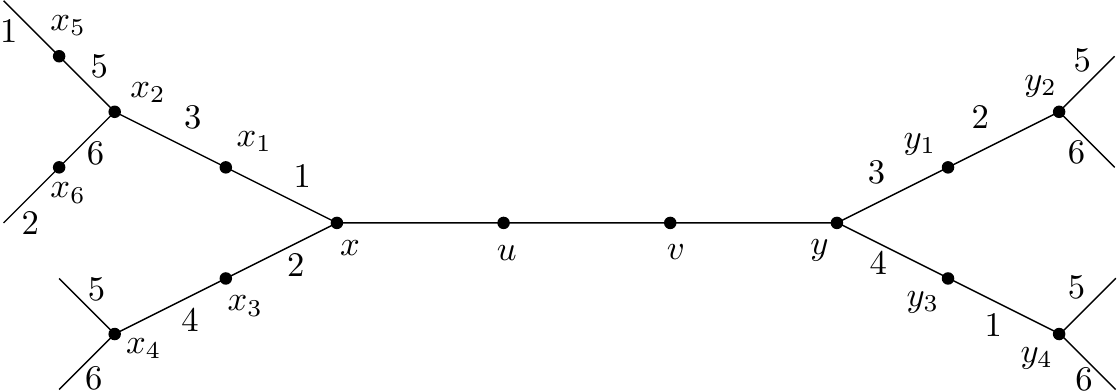}
				$$
				\caption{The labeling of the vertices together with a coloring of the edges in the neighborhoods of $u$ and $v$.}
				\label{fig:str:w1}
			\end{figure}				
			
		Observe that there is another possible nonisomorphic coloring of the edges where the colors of $y_1y_2$ and $y_3y_4$ are interchanged, however
		the argument is analogous. Notice also that the edges incident to the vertex $x_2$ (resp. $x_4$, $y_2$, $y_4$) distinct from $x_1x_2$
		(resp. $x_3x_4$, $y_1y_2$, $y_3y_4$) are colored by $5$ and $6$ for otherwise we color $xx_1$ (resp. $xx_3$, $yy_1$, $yy_3$) with $5$ 
		or $6$ infering that the three lists of available colors for $ux$, $uv$, and $vy$ are not equal. Observe that it immediately follows that 
		$x_2$, $x_4$, $y_2$, and $y_4$ are distinct vertices.
		
		Next, the two edges incident to $x_5$ and $x_6$ that are not incident to $x_2$, are colored by the colors $1$ and $2$, respectively. 
		Otherwise we interchange the color of the edges $x_1x_2$ and $xx_1$ or $xx_2$, respectively, again introducing nonequal sets of available colors
		for the three noncolored edges. However, in this setting, we are able to color $x_1x_2$ with the color $4$, and color the edges $ux$, $uv$, $vy$
		with the colors $3$, $5$, and $6$, respectively. Hence, we extended $\varphi$ to $G$, a contradiction.
		
		\item[$(c)$] Suppose that $C = uvwz$ is a $4$-cycle in $G$. By $(a)$ and $(b)$, precisely two vertices, say $u$ and $w$, are of degree $3$,
				while $v$ and $z$ are $2$-vertices. By the minimality, there exists a strong edge coloring $\varphi$ of $G\setminus \set{v,z}$
				with at most $6$ colors. It remains to color the four edges of $C$. Let $x$ and $y$ be the third neighbors of $u$ and $w$, respectively.
				In case when $x = y$, $G = K_{2,3}$, hence it is strongly edge colorable with at most $6$ colors.
				Otherwise, each of the noncolored edges has three colored $2$-neighbors and there are three available colors for each of them. 
				In case when the sets of available colors are not equal, $\varphi$ is easily extended. Hence, we consider the case when the same 
				colors are forbidden for the edges of $C$.
								
				Without loss of generality, we may assume that the color of $ux$ is $1$, and the second edge incident to $x$ is colored by $2$.
				Similarly, let $wy$ be colored with $3$ and the second edge $yy_1$ incident to $y$ must be colored with $2$. Our aim is to change
				the color of $yy_1$ in order to obtain nonequal sets of available colors of the edges of $C$. Since there are at most five 
				$2$-neighbors of $yy_1$, each of them must be colored by different color, otherwise, there is another color available for $yy_1$.
				But in this case, we recolor $wy$ with $2$ and $yy_1$ with $3$. Then we can color $uv$, $vw$, $wz$, $uz$ with $3$, $4$, $5$, and $6$,
				respectively. Hence, we extended $\varphi$ to $G$, a contradiction.
		\end{itemize}
	
	By the above claims, $G$ is a subcubic bipartite graph without $4$-cycles. Moreover, since there is no pair of adjacent $2$-vertices in $G$, it follows
	that the vertices of one partition are all of degree $3$, and the vertices in the second partition are all $2$-vertices. Such graphs,
	by Theorem~\ref{thm:str:bip23}, admit a strong edge coloring with at most $5$-colors, a contradiction.
\end{proof}

\paragraph{Acknowledgement.}
	The authors where supported in part by bilateral project SK-SI-0005-10 between Slovakia and Slovenia,
	by Science and Technology Assistance Agency under the contract No. APVV-0023-10 (R.~Sot\'{a}k), by Slovak VEGA 
	Grant No.~1/0652/12 (M.~Mockov\v{c}iakov\'{a}, R. Sot\'ak), and VVGS UPJ\v{S} No.~59/12-13 (M.~Mockov\v{c}iakov\'{a}), and
	by ARRS Program P1-0383 and Creative Core FISNM-3330-13-500033 (B. Lu\v{z}ar, R. Sot\'ak, R. \v{S}krekovski).

\end{document}